\documentclass[10pt]{amsart}
\usepackage{amsmath,amssymb}

\newcommand\D{{\not\! d}}

\newcommand\x{\xi}

\newcommand\ee{{\mathbf e}}

\newcommand\CC{\mathbb C}

\newcommand\RR{\mathbb R}

\newcommand\NN{\mathbb N}

\newcommand\GL{{\mathrm{GL}}}

\newcommand\diag{{\mathrm{diag}}}

\newcommand\End{\operatorname{End}}
\newcommand\Mat{\operatorname{Mat}}

\theoremstyle{plain}
\newtheorem{thm}{Theorem}[section]
\newtheorem{lem}[thm]{Lemma}

\newtheorem{prop}[thm]{Proposition}
\newtheorem{cor}[thm]{Corollary}

\newtheorem{ex}[thm]{Example}
\theoremstyle{definition}
\newtheorem{defn}[thm]{Definition}

\newtheorem{remark}{Remark}

\title[Reducibility of Matrix Weights]{Reducibility of Matrix Weights}
\author{Juan Tirao}
\address{Juan Tirao.
 Universidad Nacional de C\'or\-do\-ba. FaMAF, Ciudad Universitaria,
C\'or\-do\-ba PC:5000 Argentina.
tirao@famaf.unc.edu.ar}
\author{Ignacio Zurri\'an}
\address{Ignacio Zurri\'an.
Pontificia Universidad Cat\'olica de Chile. Facultad de Matem\'aticas, Vicu\~na Mackenna 4860,
Macul, Santiago, PC:7820436 Chile.
zurrian@famaf.unc.edu.ar}



\begin{document}

\begin{abstract}
In this paper we discuss the notion of reducibility for matrix weights and introduce a real vector space  $\mathcal C_\RR$ which encodes all information about the reducibility of $W$. In particular a weight $W$ reduces if and only if there is a non-scalar matrix $T$ such that $TW=WT^*$.  Also, we prove that reducibility can be studied by looking at  the commutant of the monic orthogonal polynomials or by looking at the coefficients of the corresponding three term recursion relation. A matrix weight may not be expressible as direct sum of irreducible weights, but it is always equivalent to a direct sum of irreducible weights. We also establish that the decompositions of two equivalent weights as sums of irreducible weights have the same number of terms and that, up to a permutation, they are equivalent.

We consider the algebra of right-hand-side matrix differential operators $\mathcal D(W)$ of a reducible weight $W$, giving its general structure. Finally,  we make a change of emphasis by considering reducibility of polynomials, instead of reducibility of matrix weights.
\end{abstract}

\maketitle
\begin{section}{Introduction}
  The subject of matrix valued orthogonal polynomials was introduced by M.G. Krein more than sixty years ago, see \cite{K49,K71}. In the scalar valued context, the orthogonal polynomial satisfying second order differential equations played a very important role in many areas of mathematics and its applications.  In \cite{D97}, A.J. Dur\'an started the study of matrix weights whose Hermitian bilinear form admits a symmetric second order matrix differential operator, following similar considerations by S. Bochner in the scalar case in \cite{B29}.

Motivated by previous works on the theory of matrix valued spherical functions,
see \cite{T77}, one finds in \cite{GPT02a,GPT03,G03} the first examples of classical pairs considered in \cite{D97}. In \cite{DG04} one can find the development of a method that leads to introduce some other examples of classical pairs.
In the last years, one can see a growing number of papers devoted to
different aspects of this question. For some of these recent papers, see
\cite{T03,DG05a,
CG05,GPT05,M05,PT07a}
as well as
\cite{GIV11,I11,KvPR12,PTZ14,TZ14,CI14,vPR14,AKR15}, among others.

 When one is in front of a concrete matrix weight one should ask himself: is it a new example or is it equivalent to a direct sum of known weights of smaller dimensions?
In this work we prove that a weight $W$ is reducible to weights of smaller size if and only if there exists a non-scalar constant matrix $T$ such that $TW=WT^*$, see Theorem \ref{reducibility1}. We also study the form in which a matrix weight reduces. In particular, we show that this reduction is unique.

Given a matrix weight $W$ and a sequence of matrix orthogonal polynomials $\{Q_n\}_{n\in\mathbb N_0}$, it is natural to consider the matrix differential operators $D$ such that $Q_n$ is an eigenfunction of $D$ for every $n\in\mathbb N_0$. The set of these operators is a noncommutative $*$-algebra $\mathcal{D}(W)$ (see \cite{CG06,GT07,T11,Z15}).
In Section \ref{alg} we study the structure of the algebra $\mathcal D(W)$ for a reducible weight, see Proposition \ref{prop}.

Finally, we give some irreducibility criteria, showing that one can decide the irreducibility of a weight by also considering  the commutant of the monic orthogonal polynomials $\{P_n\}_{n\in\mathbb N_0}$ (see Corollary \ref{cor}), or by looking at the commutant of the coefficients of the three term recursion relations satisfied by  them (see Theorem \ref{ab}).

 We prove in Theorem \ref{scalar} that if there is a sequence of orthogonal polynomials with respect to an irreducible weight  $W$ and with respect to a weight $V$, both with bounded support, then we have $W=\lambda V$, with $\lambda>0$.  This is not true when the support is not bounded, even more, one can have a sequence of polynomials which are orthogonal with respect to two weights $W$ and $V$, being $W$ irreducible and $V$ reducible (see Example \ref{mom}). Thus, at the end of the paper we make a change of emphasis by considering reducibility of polynomials, instead of reducibility of matrix weights.

It is known that the classical scalar orthogonal polynomials do not satisfy linear differential equations of  odd  order. Even more, the algebras $\mathcal D(W)$ are generated by one operator of order two, see \cite{M05}. The situation in the matrix context is very different, see \cite{DdI08}. Nevertheless, deciding whether there is or there is not an irreducible weight $W$, such that the algebra $\mathcal D(W)$ contains an operator of first order, is a  problem that is still unsolved. We believe that  the results in the present paper may be useful for giving an answer  in the future.
\end{section}
 \subsection*{Acknowledgement}
This paper was partially supported by CONICET PIP 112-200801-01533 and by the programme Oberwolfach Leibniz Fellows 2015.
\begin{section}{Reducibility of Matrix Weights}
	
An inner product space will be a finite dimensional complex vector space $V$ together with a specified inner product $\langle\cdot,\cdot\rangle$.
If $T$ is a linear operator on $V$ the adjoint of $T$ will be denoted $T^*$.

By a weight $\tilde W= \tilde W(x)$ of linear operators on $V$ we mean
an integrable function on an interval $(a,b)$ of the real line, such that $ \tilde W(x)$
is a (self-adjoint) positive semidefinite operator on $V$ for all $x\in (a,b)$, which
is positive definite almost everywhere and has finite moments of all orders: for all
$n\in\NN_0$ we have
$$\int_a^b x^n\tilde W(x)\,dx\in\End(V).$$

More generally we could assume that $\tilde W$ is a Borel measure on the real line
of linear operators on $V$, such that: $\tilde W(X)\in\End(V)$ is positive semidefinite for
any Borel
set $X$, $\tilde W$ has finite moments of any order, and $\tilde W$ is nondegenerate, that
is for P in the polynomial ring $\End(V)[x]$
$$(P,P)=\int_\RR P(x)\,d\tilde W(x)P(x)^*=0$$
only when $P=0$. In any case the size of $\tilde W$ is the dimension of $V$.

If we select an orthonormal basis $\{\ee_1,\dots,\ee_N\}$ in $V$, we may represent each
linear operator $\tilde W(x)$ on $V$ by a square matrix $W(x)$ of size $N$, obtaining
a weight matrix of size $N$ on the real line. By this we mean a complex
$N\times N$-matrix valued integrable function on the interval $(a,b)$
such that $W(x)$ is a (self-adjoint) positive semidefinite complex matrix for all
$x\in (a,b)$, which
is positive definite almost everywhere and with finite moments of all orders: for all
$n\in\NN_0$ we have
$$M_n=\int_a^b x^n W(x)\,dx\in\Mat_N(\CC).$$

Of course, any result involving weights of linear operators implies an analogous result for matrix weights and viceversa. Along the paper, we shall work  either with weights of linear operators or with matrix weights, depending on what we consider better for reader's  understanding.

\begin{defn}
Two weights $\tilde W= \tilde W(x)$ and $\tilde W'= \tilde W'(x)$ of linear operators
on $V$ and $V'$, respectively, defined on the same interval are {\it equivalent},
$\tilde W\sim\tilde W'$, if there exists an isomorphism $T$ of $V$ onto $V'$ such that
$\tilde W'(x)=T\tilde W(x)T^*$ for all $x$.
Moreover, if  $T$ is unitary then they are {\it unitarily equivalent}.

Two matrix weights $W=W(x)$ and $W'=W'(x)$ of the same size and defined on the same interval
are {\it equivalent}, $W\sim W'$, if there exists a nonsingular matrix
$M$ such that $W'(x)=MW(x)M^*$ for all $x$.
Moreover, if the matrix $M$ is unitary then they are {\it unitarily equivalent}.
\end{defn}
Let $S$ be an arbitrary set of elements. By an $S$-module on a field $K$ we mean a pair $(W,V)$ formed by
a finite dimensional vector space $V$ over $K$ and a mapping $W$ which assigns to every element $x\in S$ a linear
operator $W(x)$ on $V$. It follows that an $S$-module is an additive group with two domains of
operators, one being the field $K$ and the other one being the set $S$.

In particular, if $V$ is an inner product space and $W=W(x)$ is a weight of linear operators on $V$,
then $(W,V)$ is an $S$-module, where $S$ is the support of $W$. It is important to notice that
the equivalence of weights of linear operators does not coincide with the notion of isomorphism
among $S$-modules.
If $W=W(x)$ is a weight of linear operators on $V$, we say that $W$ is the orthogonal direct sum
$W=W_1\oplus\cdots\oplus W_d$ if $V$ is the orthogonal direct sum $V=V_1\oplus\cdots\oplus V_d$
where each subspace $V_i$ is $W(x)$-invariant for all $x$ and $W_i(x)$ is the restriction of $W(x)$
to $V_i$.

\begin{defn}
We say  that a weight $\tilde W=\tilde W(x)$ of linear operators on $V$ reduces to smaller-size
weights, or simply {\it reduces},  if $\tilde W$ is equivalent to a direct sum $W'=W'_1\oplus\cdots\oplus W'_N$ of
orthogonal smaller dimensional weights.

We say  that a matrix weight $W=W(x)$ reduces to smaller-size  if its corresponding weight $\tilde W=\tilde W(x)$ of linear operators reduces.
\end{defn}

\begin{thm}\label{ediagonal} A matrix weight $W$ reduces to scalar weights if and only if there exists a positive definite matrix $P$ such that for all $x,y$ we have
\begin{equation}\label{P}
W(x)PW(y)=W(y)PW(x).
\end{equation}
\end{thm}
\begin{proof} Suppose that $W(x)=M\Lambda(x)M^*$, $\Lambda=\Lambda(x)$ a diagonal weight. Then
\begin{equation*}
\begin{split}
W(x)(MM^*)^{-1}W(y)&=(M\Lambda(x)M^*)(MM^*)^{-1}(M\Lambda(y)M^*)=M\Lambda(x)\Lambda(y)M^*\\
&=M\Lambda(y)\Lambda(x)M^*=(M\Lambda(y)M^*)(MM^*)^{-1}(M\Lambda(x)M^*)\\
&=W(y)(MM^*)^{-1}W(x).
\end{split}
\end{equation*}
Hence $P=(MM^*)^{-1}$ is a positive definite matrix which satisfies \eqref{P}.

Conversely, assume \eqref{P}. First of all take $x_0$ such that $W(x_0)$ is nonsingular. Let $A$
be a symmetric positive definite matrix such that $A^2=W(x_0)$.
Take $W'(x)=A^{-1}W(x)A^{-1}$ and $P' =APA$. Then $W'$ is equivalent
to $W$, $P'$ is positive definite and
$W'(x)P'W(y)=W'(y)P'W'(x)$ for all $x,y$, and
$W'(x_0)=I$. Thus without loss of generality we may assume from the start that $W(x_0)=I$.

Now the hypothesis implies that
$W(x)P=PW(x)$ for all $x$. Let $E$ be any eigenspace of P. Then  $E$ is
$W(x)$-invariant. Hence $W(x)$ and $W(y)$ restricted to $E$ commute by \eqref{P} and are self-adjoint.
Therefore $W(x)$  restricted to $E$ are simultaneously diagonalizable for all $x$, through a
unitary operator (cf. \cite{HK65}, Corollary to Theorem 30, p. 283). Since this happens for all
eigenspaces of $P$ and they are orthogonal we have proved that $W(x)$ is unitarily
equivalent to a diagonal weight, hence $W$ reduces to scalar weights. This completes the
proof of the theorem for a matrix weight. \qed
\end{proof}

\begin{thm}\label{uediagonal} The following conditions are equivalent:

(i) A matrix weight $W$ is unitarily equivalent to a diagonal weight;

(ii) For all $x,y$ we have W(x)W(y)=W(y)W(x);

(iii) There exists a positive definite matrix $P$ such that for all $x,y$ we have $W(x)PW(y)=W(y)PW(x)$ and $W(x)P=PW(x)$.
\end{thm}
\begin{proof}

(i) implies (ii) If $U$ is unitary and $W(x)=U\Lambda(x)U^*$ where $\Lambda(x)$ is diagonal for all $x$,
then $W(x)W(y)=(U\Lambda(x)U^*)(U\Lambda(y)U^*)=U\Lambda(x)\Lambda(y)U^*=U\Lambda(y)\Lambda(x)U^*=
(U\Lambda(y)U^*)(U\Lambda(x)U^*)=W(y)W(x)$.

(ii) implies (iii) It is obvious.

(iii) implies (i) Let $E$ be any eigenspace of P. Then  $E$ is
$W(x)$-invariant. Hence $W(x)$ and $W(y)$ restricted to $E$ commute and are self-adjoint.
Therefore $W(x)$  restricted to $E$ are simultaneously diagonalizable for all $x$, through a
unitary operator (cf. \cite{HK65}, Corollary to Theorem 30, p. 283). Since this happens for all
eigenspaces of $P$ and they are orthogonal we have proved that $W(x)$ is unitarily
equivalent to a diagonal weight, hence $W$ reduces to scalar weights. This completes the
proof of the theorem.  \qed
\end{proof}

The following corollary can also be found in \cite[p. 463]{DG04}.
\begin{cor}\label{equivalencia} Let $W=W(x)$ be a matrix weight  such that $W(x_0)=I$ for some
$x_0$ in the support of $W$. Then $W$ reduces to scalar weights if and only if $W$ is unitarily
equivalent to a diagonal weight.
\end{cor}
\begin{proof} Let us consider first the case of a matrix weight. It is obvious from the definition that if $W$ is unitarily equivalent to a diagonal weight,
then $W$ reduces to scalar weights. Conversely, let us assume that $W$ reduces to scalar weights.
By Theorem \ref{ediagonal} there exists a positive definite matrix $P$ such that $W(x)PW(y)=W(y)PW(x)$
for all $x,y$. If we put $y=x_0$ we get $W(x)P=PW(x)$ for all $x$. Then Theorem \ref{uediagonal}
implies that $W$ is unitarily equivalent to a diagonal weight. \qed
\end{proof}

\begin{ex}\label{example1} Let $W(x)=\left(\begin{matrix} x^2+x&x\\x&x\end{matrix}\right)$ supported in $(0,1)$. Then
$$W(x)=\left(\begin{matrix}1&1\\0&1\end{matrix}\right)\left(\begin{matrix} x^2&0\\0&x\end{matrix}\right)\left(\begin{matrix}1&0\\1&1\end{matrix}\right).$$
Therefore $W$ reduces to scalar weights and $W(x)PW(y)=W(y)PW(x)$ for all $x,y\in(0,1)$
with
$$P=(MM^*)^{-1}=\left[\left(\begin{matrix}1&1\\0&1\end{matrix}\right)\left(\begin{matrix}1&0\\1&1
\end{matrix}\right)\right]^{-1}=\left(\begin{matrix}1&-1\\-1&2\end{matrix}\right).$$
However $W$ is not unitarily equivalent to a diagonal weight. In fact
\begin{equation*}
\begin{split}
W(x)W(y)&=\left(\begin{matrix} x^2+x&x\\x&x\end{matrix}\right)\left(\begin{matrix} y^2+y&y\\y&y\end{matrix}\right)\\
&=\left(\begin{matrix}(x^2+x)(y^2+y)+xy&(x^2+2x)y\\(y^2+2y)x&2xy
\end{matrix}\right).
\end{split}
\end{equation*}
Since the entries $(1,2)$ and $(2,1)$ are not symmetric in $x$ and $y$, $W(x)W(y)\ne W(y)W(x)$ for all $x,y$.
\end{ex}

\begin{defn} We say that a matrix weight $W$ is the direct sum of the matrix weights $W_1,W_2,\dots,W_j$, and write $W=W_1\oplus W_2\oplus\cdots\oplus W_j$, if
$$W(x)=\left(\begin{matrix}W_1(x)&{\bf0}&\cdot&\cdot&\cdot&{\bf0}\\{\bf0}&W_2(x)&\cdot&\cdot&\cdot&{\bf0}\\
\cdot&\cdot&&&&\cdot\\\cdot&\cdot&&&&\cdot\\\cdot&\cdot&&&&\cdot\\
{\bf0}&{\bf0}&\cdot&\cdot&\cdot&W_j(x)\\\end{matrix}\right).$$
\end{defn}

\begin{thm} \label{reducibility1}Let $W=W(x)$ be a matrix weight function with support $S$. Then the following conditions are equivalent:

\noindent(i) $W$ is equivalent to a direct sum of matrix weights of smaller size;

\noindent (ii) there is an idempotent matrix $Q\ne 0,I$ such that
$QW(x)=W(x)Q^*$ for all x;

\noindent (iii) the  space $\mathcal C_\RR\equiv\{T\in\Mat_N(\CC):TW(x)=W(x)T^*\; \text{for a.e.}\; x\in S\}\supsetneq\RR I$.
\end{thm}
\begin{proof}

(i) implies (ii) Suppose that $\tilde W(x)=MW(x)M^*=\tilde W_1(x)\oplus\tilde W_2(x)$ where $\tilde W_1$ is of size $k$ and $\tilde W_2(x)$ is of size $n-k$. Let $P$ be the orthogonal projection onto the subspace of $\CC^n$ generated by $\{e_1,\dots,e_k\}$. Then $P\tilde W(x)=\tilde W(x)P$. Hence $P(MW(x)M^*)=(MW(x)M^*)P$ and
$$(M^{-1}PM)W(x)=W(x)(M^*P(M^*)^{-1})=W(x)(M^{-1}PM)^*.$$
Therefore if we take $Q=M^{-1}PM$, then $Q$ is idempotent, $Q\ne 0,I$ and for all $x$ $QW(x)=W(x)Q^*$.

(ii) implies (iii) Clearly $\mathcal C_\RR$ is a real vector space such that $I\in \mathcal C_\RR$. Now the implication is obvious.

(iii) implies (i) Our first observation is that if $T\in\mathcal C_\RR$ then its eigenvalues are real. In fact, by changing $W$ by an equivalent weight function we may assume that $T$ is in Jordan canonical form. Thus $T=\bigoplus_{1\le i\le s} J_i$ is the direct sum of elementary Jordan matrices $J_i$ of size $d_i$ with characteristic value $\lambda_i$ of the form
$$J_i=\left(\begin{matrix}\lambda_i&0&\cdot&\cdot&\cdot&0&0\\1&\lambda_i&\cdot&\cdot&\cdot&0&0\\
\cdot&\cdot&&&&\cdot&\cdot\\\cdot&\cdot&&&&\cdot&\cdot\\\cdot&\cdot&&&&\cdot&\cdot\\
0&0&\cdot&\cdot&\cdot&\lambda_i&0\\0&0&\cdot&\cdot&\cdot&1&\lambda_i\\\end{matrix}\right).$$
Let us write $W(x)$ as an $s\times s$-matrix of blocks $W_{ij}(x)$ of $d_i$-rows and $d_j$-columns. Then, by hypothesis, we have $J_iW_{ii}(x)=W_{ii}(x)J_i^*$. Thus, if $1\le k\le n$ is the index corresponding to the first row of $W_{ii}(x)$, we have $\lambda_i w_{kk}(x)=w_{kk}(x)\bar\lambda_i$. But there exists $x$ such that $W(x)$ is positive definite
and so $w_{kk}(x)>0$. Therefore $\lambda_i=\bar\lambda_i$.

Another important property of the real vector space $\mathcal C_\RR$ is the following one: if $T\in\mathcal C_\RR$ and $p\in\RR[x]$ then $p(T)\in\mathcal C_\RR$.

Let $L$  and $N$ be, respectively, the semi-simple and the nilpotent parts of $T$. The minimal polynomial of $T$ is $(x-\lambda_1)^{r_1}\cdots (x-\lambda_j)^{r_j}$ where $\lambda_1,\dots,\lambda_j$ are the different characteristic values of $T$ and $r_i$ is the greatest dimension of the elementary Jordan matrices with characteristic value $\lambda_i$. A careful look at the proof of Theorem 8, Chapter 7 of \cite{HK65} reveals that $L$ and $N$ are real polynomials in $T$. Therefore $L ,N\in\mathcal C_\RR$.

The minimal polynomial of $L$ is $p=(x-\lambda_1)\cdots(x-\lambda_j)$. Let us consider the Lagrange polynomials
$$p_k=\prod_{i\ne k}\frac{(x-\lambda_i)}{(\lambda_i-\lambda_k)}.$$
Since $p_k(\lambda_i)=\delta_{ik}$, and $ L=L^* $ it follows that $P_i=p_i( L )$ is the orthogonal projection of $\CC^n$ onto the $\lambda_i$-eigenspace of $L$. Therefore $P_i\in\mathcal C_\RR$.

Let us now consider the nilpotent part $N$ of $T$: $N$ is the direct sum of the nilpotent parts $N_i$ of each elementary Jordan block $J_i$, i.e.
$$N_i=\left(\begin{matrix}0&0&\cdot&\cdot&\cdot&0&0\\1&0&\cdot&\cdot&\cdot&0&0\\
\cdot&\cdot&&&&\cdot&\cdot\\\cdot&\cdot&&&&\cdot&\cdot\\\cdot&\cdot&&&&\cdot&\cdot\\
0&0&\cdot&\cdot&\cdot&0&0\\0&0&\cdot&\cdot&\cdot&1&0\\\end{matrix}\right).$$
Since $N\in\mathcal C_\RR$ we have $N_iW_{ii}(x)=W_{ii}(x)N_i^*$. If $N_i$ were a matrix of size larger than one, and if $1\le k\le n$ is the index corresponding to the first row of $W_{ii}(x)$, we would have $w_{kk}(x)=0$ for all $x$, which is a contradiction. Therefore all elementary Jordan matrices of $T$ are one dimensional, hence $N=0$ and $T=L$.

If $\mathcal C_\RR\supsetneq\RR I$, then there exists $T\in\mathcal C_\RR$ with $j>1$ different eigenvalues. Let $M\in \GL(N,\CC)$ such that $L=MTM^{-1}$ be a diagonal matrix and let $P_i$, $1\le i\le j$  be the orthogonal projections onto the eigenspaces of $L$. Let $\tilde W(x)=MW(x)M^*$. Then $I=P_1+\cdots+P_j$, $P_rP_s=P_sP_r=\delta_{rs}P_r, 1\le r,s\le j$ and $P_r^*=P_r$, $P_r\tilde W(x)=\tilde W(x)P_r$ for all $1\le r\le j$. Therefore,
\begin{equation*}
\begin{split}
\tilde W(x)&=(P_1+\cdots+P_j)\tilde W(x)(P_1+\cdots+P_j)=\sum_{1\le r,s\le j}P_r\tilde W(x)P_s\\
&=P_1\tilde W(x)P_1+\cdots+P_j\tilde W(x)P_j=\tilde W_1(x)\oplus\cdots\oplus\tilde W_j(x),
\end{split}
\end{equation*}
completing the proof that (iii) implies (i). Hence the theorem is proved. \qed
\end{proof}

\
For future reference we introduce the following definition.
\begin{defn}\label{CR} If $W=W(x)$ is a matrix weight function with support $S$, then we define the real vector space
$$\mathcal C_\RR=\{T\in\Mat_N(\CC):TW(x)=W(x)T^*\; \text{for a.e.}\; x\in S\}.$$
\end{defn}

Any matrix weight $W$ is equivalent to a matrix weight $\tilde W$ with the order zero moment $\tilde M_0=I$. In fact, it suffices to take $\tilde W=M_0^{-1/2}WM_0^{-1/2}$.

\begin{prop}\label{comm} Let $W$ be a weight matrix with support $S$ such that $M_0=I$. Then its real vector space $\mathcal C_\RR$ is a real form of the commutant algebra
$$\mathcal C=\{T\in\Mat_N(\CC):TW(x)=W(x)T\; \text{for a.e.}\; x\in S\}.$$
\end{prop}
\begin{proof} The hypothesis implies right away that $T=T^*$, hence the proposition follows. \qed.
\end{proof}

It is worth to state the following immediate corollary of Theorem \ref{reducibility1} which resembles a Schur's lemma.

\begin{cor}\label{Schur} Let $W=W(x)$ be a matrix  weight . Then $W$ is irreducible if and only if
its  real vector space $\mathcal C_\RR=\RR I$.
\end{cor}

The following result can also be found in a recent work announced in \cite{KR15}.
\begin{cor} A matrix weight reduces if and only if the commutant of $\tilde W=M_0^{-1/2}WM_0^{-1/2}$ contains non-scalar matrices.
\end{cor}

Let $(W,V)$ be an abstract $S$-module. Then $(W,V)$ is said to be simple if it is of positive
dimension and if the only invariant subspaces of $V$ are $\{0\}$ and $V$. An $S$-module is called semi-simple if it can be represented as a sum of simple submodules.

Instead, when $W=W(x)$ is a matrix weight (resp. a weight of linear operators on V) one says that $W$ is irreducible when it is not equivalent to a direct sum of matrix (resp. an orthogonal direct sum of operator) weights of smaller size. Furthermore an operator weight will be called simple if has no proper invariant subspace.

A semi-simple abstract $S$-module $(W,V)$ can be represented as the direct sum
$V=V_1\oplus\cdots\oplus V_j$ of a finite collection $\Phi=\{V_i\}$ of simple $S$-submodules. Moreover, if we have
a representation of this kind, and if $V'$ is any invariant subspace of $V$, then there exists a
subcollection $\Phi_0$ of $\Phi$ such that $V$ is the direct sum of $V'$ and of the sum of the submodules
belonging to $\Phi_0$.

Conversely, let $V$ be an $S$-module which has the following property: if $V'$ is any invariant subspace of $V$,
there exists an invariant subspace $V''$ such that $V=V'\oplus V''$. Then $V$ is semi-simple (cf. Propositions 1,\,2 in Ch. VI, $\S$I of \cite{C99}).

\begin{prop} Let $W=W(x)$ be a matrix weight with support $S$. Then the $S$-module
$(W,V)$ is semi-simple.
\end{prop}
\begin{proof} Let $V'$ be any invariant subspace of $V$ and let $V''$ be its orthogonal complement. Then, since $W(x)$ is self-adjoint for all $x$, $V''$ is invariant. Therefore the $S$-module
$(W,V)$ is semi-simple. \qed
\end{proof}

Let $V=V_1\oplus\cdots\oplus V_j=V'_1\oplus\cdots\oplus V'_{j'}$ be two representations of an abstract semi-simple
$S$-module $V$ as a direct sum of simple submodules. Then we have $j=j'$ and there exists a permutation $\sigma$
of the set $\{1,\dots, j\}$ such that $V_i$ is isomorphic to $V'_{\sigma(i)}$ for all $1\le i\le j$ (cf. Propositions 3,
in Ch. VI, $\S$I of \cite{C99}). Clearly, this uniqueness result can be generalized to the following one: if
$V=V_1\oplus\cdots\oplus V_j$ and $V'=V'_1\oplus\cdots\oplus V'_{j'}$ are two isomorphic $S$-modules represented
as direct sums of simple submodules, then we have $j=j'$ and there exists a permutation $\sigma$
of the set $\{1,\dots, j\}$ such that $V_i$ is isomorphic to $V'_{\sigma(i)}$ for all $1\le i\le j$.

This last statement is not true for matrix weights or operator weights on an inner product space $V$, because, as we
pointed out at the beginning, the equivalence among weights has a different meaning than the isomorphism of $S$-modules.
In fact, Example \ref{example1}  illustrates   this: the matrix weight
$$W(x)=\left(\begin{matrix} x^2+x&x\\x&x\end{matrix}\right)$$
has no invariant subspace of $\CC^n$, i.e. it is simple, but it is equivalent to
$$W'(x)=\left(\begin{matrix} x^2&0\\0&x\end{matrix}\right)$$
 which is the direct sum of two scalar weights.
Nevertheless the following holds.

\begin{thm}\label{equivalencia2} Let $W=W_1\oplus\cdots\oplus W_j$ and $W'=W'_1\oplus\cdots\oplus W'_{j'}$ be representations of the operator weights $W$ and $W'$ as orthogonal direct sums of simple (irreducible) weights. If $W$ and $W'$ are unitarily equivalent, then we have $j=j'$ and there exists a permutation $\sigma$
of the set $\{1,\dots, j\}$ such that $W_i$ is unitarily equivalent to $W'_{\sigma(i)}$ for all $1\le i\le j$.
\end{thm}
\begin{proof} It is clear that it is enough to consider the case $W=W'$.
We shall construct the permutation $\sigma$. Suppose that $\sigma(i)$ is already defined for $i<k$ (where $1\le k\le j$)
and has the following properties: a) $\sigma(i)\ne\sigma(i')$ for $i<i'<k$; b) $V'_i$ is unitarily equivalent to
$V'_{\sigma(i)}$ for $i<k$; c) we have the orthogonal direct sum
$$V=\bigoplus_{i<k}V'_{\sigma(i)}\oplus\; \bigoplus_{i\ge k}V_i.$$
We consider the invariant subspace
$$E=\bigoplus_{i<k}V'_{\sigma(i)}\oplus\; \bigoplus_{i> k}V_i.$$
Let $E^\perp$ be the orthogonal complement of $E$ which is, by Propositions 1 in Ch. VI, $\S$I of \cite{C99} mentioned above,
the direct sum of a certain number of the spaces $V'_i$. On the other hand $E^\perp$ is unitarily isomorphic to $V/E$, i.e to $V_k$.
It follows that $E^\perp$ is simple and therefore $E^\perp$ is one of the $V'_{i}$, say $E^\perp=V'_{i_0}$.
Since $V'_{\sigma(i)}\subset E$ for $i<k$, we have $i_0\ne\sigma(i)$ for $i<k$. We define $\sigma(k)=i_0$. It is clear that the function $\sigma(i)$ now defined for $i<k+1$, satisfies  conditions a), b), c) above, with $k$ replaced by $k+1$.

Because we can define the injective function $\sigma$ on the set $\{1,\dots,j\}$ we must have $j'\ge j$. Since the two decompositions play symmetric  roles, we also have $j\ge j'$, hence $j=j'$. The theorem is proved. \qed
\end{proof}

If we come back to our Example \ref{example1} we realize that a matrix weight may not be expressible  as  direct sum of irreducible matrix weights. But such a weight is equivalent to one that is the direct sum of  irreducible weights. Taking into account this fact and Theorem \ref{reducibility1} we make the following definition.

\begin{defn} We say that a matrix weight (resp. an operator weight) is completely reducible if it is equivalent to a direct sum of irreducible matrix weights (resp. an orthogonal direct sum of irreducible operators). \end{defn}

Observe that Theorem \ref{reducibility1} implies that every  weight is completely reducible.

\

\begin{thm}\label{unicidad}
Let $W=W_1\oplus\cdots\oplus W_j$ and $W'=W'_1\oplus\cdots\oplus W'_{j'}$ be representations of the operator weights $W$ and $W'$ as orthogonal direct sums of irreducible weights. If $W$ and $W'$ are equivalent, then we have $j=j'$ and there exists a permutation $\sigma$ of the set $\{1,\dots, j\}$ such that $W_i$ is equivalent to $W'_{\sigma(i)}$ for all $1\le i\le j$.
\end{thm}
\begin{proof}
Modulo unitary equivalence we may assume from the beginning that $W$ and $W'$ are
matrix weights of the same size. Let $W'(x)=MW(x)M^*$ for all $x$, with $M$ a
nonsingular matrix. We may write $M=UP$ in a unique way with $U$ unitary and $P$
positive definite, and $P=VDV^*$ where $V$ is unitary and $D$ is a positive diagonal
matrix. Then $W'(x)=(UV)D(V^*W(x)V)D(UV)^*$. Modulo unitary equivalences we may
assume that $W'(x)=DW(x)D$. If we write $D=D_1\oplus\cdots\oplus D_j$ where $D_i$
is a diagonal matrix  block of the same size as the matrix block $W_i$. Then
$W'=(D_1W_1D_1)\oplus\cdots\oplus(D_jW_jD_j)$.
From the hypothesis we also have the representation of $W'=W'_1\oplus\cdots\oplus W'_{j'}$
as an orthogonal direct sum of irreducible weights.
Now we are ready to apply Theorem \ref{equivalencia2} to conclude that $j=j'$
and that there exists a permutation $\sigma$ such that $D_iW_iD_i\approx W'_{\sigma(i)}$ for all $1\le i\le j$. Hence $W_i\sim W'_{\sigma(i)}$ for all $1\le i\le j$, completing the proof of the theorem. \qed
\end{proof}

\begin{thm}\label{reducibilitydegree}
Let $W=W(x)$ be an operator weight equivalent to an orthogonal direct sum $W_1\oplus\cdots\oplus W_d$ of
irreducible weights.  If $j(T)$ is the number of distinct eigenvalues of $T\in\mathcal C_\mathbb{R}$,
then $d=\max\{j(T):T\in\mathcal C_\mathbb{R}\}$. Moreover, $W$ is equivalent to a matrix weight
$W'=W'_1\oplus\cdots\oplus W'_d$ where $W'_i$ is the restriction of $W'$ to one of the
eigenspaces of a diagonal matrix $D\in\mathcal C_\mathbb{R}$. Besides  $j(T)$ is the degree of the minimal polynomial of $T$.
\end{thm}
\begin{proof}
Suppose that $W=W_1\oplus\cdots\oplus W_d$ and let $P_1,\dots,P_d$ be the corresponding
orthogonal projections. Define $T=\lambda_1 P_1+\cdots+\lambda_d P_d$ with
$\lambda_1,\dots,\lambda_d$ all different. Then clearly $T\in\mathcal C_\mathbb{R}$. Hence
$d\le\max\{j(T):T\in\mathcal C_\mathbb{R}\}$.

Conversely, let $T\in\mathcal C_\mathbb{R}$ such that $j(T)=\max\{j(T):T\in\mathcal C_\mathbb{R}\}$.
Modulo unitary equivalence we may assume that $W$ is a matrix weight. In the proof
of Theorem \ref{reducibility1} we established that $T$ is semi-simple. Thus we may write
$T=A^{-1}DA$ with $D$ a diagonal matrix. Let $W'(x)=AW(x)A^*$. Then $W'\sim W$ and
$D\in\mathcal C_\mathbb{R}(W')$. We may assume that $D=D_1\oplus\cdots\oplus D_{j(T)}$ where $D_i$ is the
Jordan diagonal block corresponding to the eigenvalue $\lambda_i$ of $D$. Then
$W'=W'_1\oplus\cdots\oplus W'_{j(T)}$ where the size of the block $W'_i$ is equal to the
size of the block $D_i$, for all $1\le i\le j(T)$, because $DW'(x)=W'(x)D$. If some $W'_i$ were not irreducible we could replace it, modulo equivalence,  by a direct sum of matrix irreducible weights. Thus there exists  a matrix weight $W''\sim W'$ such that $W''=W''_1\oplus\cdots\oplus W''_j$ is a direct sum of matrix irreducible weights with $j(T)\le j$.

By hypothesis $W=W_1\oplus\cdots\oplus W_d$. Hence Theorem \ref{unicidad} implies that
$d=j$. Therefore $d=j(T)$ and the $W'_i$ are in fact irreducible. Moreover,  there exists a permutation
$\sigma$ of the set $\{1,\dots,d\}$ such that $W'_i=W_{\sigma(i)}$. The theorem is proved. \qed
\end{proof}

\

\end{section}

\begin{section}{The algebra $\mathcal D(W)$ of a reducible weight}\label{alg}

We come now to the notion of a differential operator with matrix
coefficients acting on matrix valued polynomials, i.e. elements of
$\Mat_N(\CC)[x]$. These operators could be made to act on our functions either
on the left or on the right. One finds a discussion of these two
actions in \cite{D97}. The conclusion is that if one wants to
have matrix weights that do not reduce to scalar weights and
that have matrix polynomials as their eigenfunctions, one should
settle for right-hand-side differential operators. We agree now to
say that $D$ given by
$$D=\sum_{i=0}^s \partial^iF_i(x),\quad\quad \partial=\frac{d}{dx},$$
acts on $Q(x)$ by means of
$$ QD = \sum_{i=0}^s \partial^i(Q)(x)F_i(x)$$

Given a sequence of matrix orthogonal polynomials $\{Q_n\}_{n\in\mathbb{N}_0}$ with respect to a weight matrix $W=W(x)$, we introduced in \cite{GT07}  the algebra $D(W)$ of all right-hand side
differential operators with matrix valued coefficients that have the
polynomials $Q_w$ as their eigenfunctions.  Thus
\begin{equation}\label{D}
\mathcal D(W)=\{D:Q_nD=\Gamma_n(D)Q_n,\; \Gamma_n(D)\in\Mat_N(\CC)\;\text{for
all}\;n\in\mathbb{N}_0\}.
\end{equation}

\bigskip

The definition of $\mathcal D(W)$ depends only on the weight matrix $W$ and not on the sequence
$\{Q_n\}_{n\in\mathbb{N}_0}$.

\begin{remark}

If $W'\sim W$, say $W'=MWM^*$, then the map $D\mapsto MDM^{-1}$ establishes an isomorphism  between the algebras $\mathcal D(W)$ and $\mathcal D(W')$. In fact, if $\{Q_n\}_{n\in\mathbb{N}_0}$ is a sequence of matrix orthogonal polynomials with respect to $W$, then $\{Q'_n=MQ_nM^{-1}\}_{n\in\mathbb{N}_0}$ is a sequence of matrix orthogonal polynomials with respect to $W'$. Moreover, if $Q_nD=\Gamma_n(D)Q_n$, then
$$Q'_n(MDM^{-1})= (M\Gamma_n(D)M^{-1})Q'_n.$$
Hence $MDM^{-1}\in\mathcal D(W')$ and $\Gamma_n(D')=M\Gamma_n(D)M^{-1}$.
\end{remark}

It is worth observing that each algebra $\mathcal D(W)$ is a subalgebra of the Weyl algebra $\mathbf D$ over $\Mat_N(\CC)$ of all linear right-hand side ordinary differential operators with coefficients in $\Mat_N(\CC)[x]$:
$$\mathbf D=\Big\{D=\sum_i\partial^i F_i: F_i\in\Mat_N(\CC)[x]\Big\}.$$

It is also interesting to introduce the subalgebra $\mathcal D$ of the Weyl algebra $\mathbf D$ defined by
$$\mathcal D=\Big\{D=\sum_i\partial^i F_i\in \mathbf D_N: \deg F_i\le i\Big\}.$$

A differential operator $D$ is symmetric if $(PD,Q)=(P,QD)$ for all $P,Q\in\Mat_N(\CC)[x]$. The set $\mathcal S(W)$ of all symmetric differential operators is a real form of $\mathcal D(W)$, i.e.,
\begin{equation}\label{real}
 \mathcal D(W)=\mathcal S(W)\oplus i\mathcal S(W),
\end{equation}
see Corollary 4.5 of \cite{GT07}.

\

In this section we will pay special attention to the algebra $\mathcal D(W)$ when the matrix weight $W=W_1\oplus W_2$. Let $N$, $N_1$ and $N_2$ be, respectively, the sizes of $W$, $W_1$ and $W_2$. If $D\in\mathbf D_N$ we will write
$$D=\begin{pmatrix} D_{11}&D_{12}\\D_{21}&D_{22}\end{pmatrix},$$
where $D_{ij}$ is a right-hand size differential operator of size $N_i\times N_j$.

\begin{prop}\label{W1W2} If $W_1$ and $W_2$ are equivalent weight matrices, then
$$\mathcal D(W_1\oplus W_2)\simeq\Mat_2(\mathcal D(W_1))\simeq\Mat_2(\CC)\otimes \mathcal D(W_1);$$
both isomorphims are canonical.
\end{prop}
\begin{proof} We first observe that by changing $W$ by an equivalent weight matrix we can assume that $W_1=W_2$. Moreover, if $\{P_n\}_{n\in\mathbb{N}_0}$ is the sequence of monic orthogonal polynomials with respect to $W_1$ then $\{\tilde P_n=P_n\oplus P_n\}_{n\in\mathbb{N}_0}$ is the sequence of monic  orthogonal polynomials with respect to $W$.

If $D=\begin{pmatrix}D_{11}&D_{12}\\D_{21}&D_{22}\end{pmatrix}\in\mathcal D(W)$, then $\tilde P_nD=\Lambda_n(D)\tilde P_n$, which in matrix notation is equivalent to $$\begin{pmatrix}P_n&0\\0&P_n\end{pmatrix}\begin{pmatrix}D_{11}&D_{12}\\D_{21}&D_{22}
\end{pmatrix}=\begin{pmatrix}\Lambda^n_{11}(D)&\Lambda^n_{12}(D)\\\Lambda^n_{21}(D)&\Lambda^n_{22}
(D)\end{pmatrix}\begin{pmatrix}P_n&0\\0&P_n\end{pmatrix}.$$
 Therefore $P_n D_{ij}=\Gamma^n_{ij}(D)P_n$ for $1\le i,j\le 2$, which implies that $D_{ij}\in\mathcal D(W_1)$.

 Conversely, given $D_{ij}\in\mathcal D(W_1)$ for $1\le i,j\le 2$, then $P_nD_{ij}=\Lambda_n(D_{ij})P_n$. If we define $\Gamma_n=\begin{pmatrix}\Lambda_n(D_{11})&\Lambda_n(D_{12})\\\Lambda_n(D_{21})&\Lambda_n(D_{22})
\end{pmatrix}$, then $\tilde P_n D=\Lambda_n \tilde P_n$. Hence, $D\in\mathcal D(W)$.

Thus, $D\in\mathcal D(W)$ if and only if $D_{ij}\in\mathcal D (W_1)$ completing  the proof of the theorem. \qed
\end{proof}

Proposition \ref{W1W2} admits the following obvious generalization.
\begin{thm} If $W_1, W_2,\dots, W_n$ are equivalent weight matrices, then
$$\mathcal D(W_1\oplus\cdots\oplus W_n)\simeq\Mat_n(\mathcal D(W_1))\simeq\Mat_n(\CC)\otimes \mathcal D(W_1);$$
both isomorphims are canonical.
\end{thm}

We move to consider the case of a direct sum of two arbitrary matrix weights $W_1$ and $W_2$ not necessarily equivalent. Let $\{P_n\}_{n\in\mathbb{N}_0}$ and $\{Q_n\}_{n\in\mathbb{N}_0}$ be, respectively, the sequences of monic orthogonal polynomials corresponding to $W_1$ and $W_2$ and let $\tilde P_n=P_n\oplus Q_n$ be the sequence of monic orthogonal polynomials of $W=W_1\oplus W_2$. If $D\in \mathcal D(W)$, then $\tilde P_n D=\Lambda_n(D)\tilde P_n$, which is equivalent to $$\begin{pmatrix}P_n&0\\0&Q_n\end{pmatrix}\begin{pmatrix}D_{11}&D_{12}\\D_{21}&D_{22}
\end{pmatrix}=\begin{pmatrix}\Lambda^n_{11}(D)&\Lambda^n_{12}(D)\\\Lambda^n_{21}(D)&\Lambda^n_{22}
(D)\end{pmatrix}\begin{pmatrix}P_n&0\\0&Q_n\end{pmatrix}.$$
Therefore, $P_n D_{11}=\Lambda^n_{11}(D)P_n$ and $Q_n D_{22}=\Lambda^n_{22}(D)Q_n$ which implies that $D_{11}\in\mathcal D(W_1)$ and that $D_{22}\in\mathcal D(W_2)$. Besides,
$P_n D_{12}=\Lambda^n_{12}(D)Q_n$ and $Q_n D_{21}=\Lambda^n_{21}(D)P_n$. Let
$$\mathbf D_{N_1\times N_2}=\Big\{D=\sum_i\partial^i F_i: F_i\in\Mat_{N_1\times N_2}(\CC)[x]\Big\},$$
$$\mathcal D_{N_1\times N_2}=\Big\{D=\sum_i\partial^i F_i\in \mathbf D_{N_1\times N_2}: \deg F_i\le i\Big\}.$$
It is also interesting to introduce
$$\mathcal D(W_1,W_2)=\{D\in\mathcal D_{N_1\times N_2}:P_n D=\Lambda_n(D)Q_n,\; \Lambda_n(D)\in\Mat_{N_1\times N_2}(\CC),\;n\in\mathbb{N}_0\},$$
$$\mathcal D(W_2,W_1)=\{D\in\mathcal D_{N_2\times N_1}:Q_n D=\Lambda_n(D)P_n,\; \Lambda_n(D)\in\Mat_{N_2\times N_1}(\CC),\;n\in\mathbb{N}_0\}.$$

When $W_1=W_2$ we have $\mathcal D(W_1,W_2)=\mathcal D(W_1)$. We also observe that $\mathcal D(W_1,W_2)$ is a left $\mathcal D(W_1)$-module and a right $\mathcal D(W_2)$-module, and that $\mathcal D(W_2,W_1)$ is a left $\mathcal D(W_2)$-module and a right $\mathcal D(W_1)$-module. Besides $\mathcal D(W_1,W_2)\mathcal D(W_2,W_1)\subset\mathcal D(W_1)$ and $\mathcal D(W_2,W_1)\mathcal D(W_1,W_2)\subset\mathcal D(W_2)$.

Clearly $D\in\mathcal D(W)$ if and only if $D_{11}\in\mathcal D(W_1)$, $D_{22}\in\mathcal D(W_2)$, $D_{12}\in\mathcal D(W_1,W_2)$ and $D_{21}\in\mathcal D(W_2,W_1)$. In other words we have proved the following proposition.
\begin{prop}\label{prop} If $W_1$ and $W_2$ are two matrix weights, then
$$\mathcal D(W_1\oplus W_2)=\begin{pmatrix}\mathcal D(W_1)&\mathcal D(W_1,W_2)\\
\mathcal D(W_2,W_1)&\mathcal D(W_2)\end{pmatrix}.$$
\end{prop}

It is worth to remark at this point that if one is interested in studying, for example, {\it classical pairs} $(W,D)$ consisting of a weight matrix $W$ and a second order symmetric differential operator $D\in\mathcal D$ or the algebra $\mathcal D(W)$, it is not enough to consider irreducible weights.

We also point out that the above proposition generalize to the case of a direct sum of $r\ge2$ matrix weights, giving a similar structure for $\mathcal D(W_1\oplus\dots\oplus W_r)$.

\subsection{An Example}
\

The pair $(W,D)$ considered here  is borrowed from 5.1 in \cite{GPT03} by taking $\alpha=\beta$. Let  $W$ be the following  matrix weight
$$W(x)=x^\alpha(1-x)^\alpha\begin{pmatrix}1&-2x+1\\-2x+1&1\end{pmatrix}=x^\alpha(1-x)^\alpha(xF_1+F_0),$$
supported in the closed interval $[0,1]$ with $\alpha>-1$.
Let us first compute the commuting space $\mathcal C=\{T\in\Mat_2(\CC):TW(x)=W(x)T^*\; \text{for all}\; x\}$ of $W$. A $2\times 2$-matrix $T=(t_{ij})_{i,j}\in\mathcal C$ if and only if $TF_1=F_1T^*$ and $TF_0=F_0T^*$, or equivalently if and only if $TI=IT^*$ and $TJ=JT^*$ where $I$ is the identity matrix and $J$ is the anti identity matrix. This immediately leads to $T=sI+tJ$, $s,t\in\RR$. Therefore, Theorem \ref{reducibility1} says that $W$ reduces to a direct sum of scalar weights and indicates how we can find them.

A nontrivial idempotent in $\mathcal C$ is obtained precisely for $s=1/2$ and $t=\pm 1/2$ Let us take T=$(I+J)/2$. The eigenvalues of a nontrivial idempotent transformation are $1$ and $0$ and the corresponding eigenspaces are $V_1=\{Tv: v\in\CC^2\}$ and $V_0=\{v-Tv: v\in\CC^2\}$. Hence the column vectors $(1,1)^t\in V_1$ and $(1,-1)^t\in V_0$. If we take  $M^{-1}=\begin{pmatrix}1&1\\1&-1\end{pmatrix}$, then $MTM^{-1}=\diag (1,0)$. From the proof of Theorem \ref{reducibility1} we know that $W'(x)=MW(x)M^*$ is a diagonal weight. In fact
\begin{equation*}
\begin{split}
W'(x)&=\frac14x^\alpha(1-x)^\alpha\begin{pmatrix}1&1\\1&-1\end{pmatrix}
\begin{pmatrix}1&-2x+1\\-2x+1&1\end{pmatrix}\begin{pmatrix}1&1\\1&-1\end{pmatrix}\\
&=x^\alpha(1-x)^\alpha\begin{pmatrix}1-x&0\\0&x\end{pmatrix}.
\end{split}
\end{equation*}
Let us consider the scalar weights $w_1(x)=x^\alpha(1-x)^{\alpha+1}$ and $w_2(x)=x^{\alpha+1}(1-x)^{\alpha}$ supported in $[0,1]$. Then $W'=w_1\oplus w_2$.

\

From 5.1 in \cite{GPT03}, by transposing, we take the following symmetric differential operator with respect to $W$:
$$D=\partial^2x(1-x)+\partial(X-xU)+V,$$
\begin{equation*}
\begin{split}
V&=\begin{pmatrix}0&0\\0&-2(\alpha+1)(1+u)\end{pmatrix},\qquad X=\begin{pmatrix}\alpha+1-\frac{u}2&\frac{u}2\\-\frac{u}2-1&\alpha+2+\frac{u}2\end{pmatrix},\\
U&=\begin{pmatrix}2(\alpha+1)-u&0\\0&2(\alpha+2)+u\end{pmatrix},
\end{split}
\end{equation*}
$u$ being a real parameter.
Therefore $D'=MDM^{-1}$ is symmetric with respect to $W'$. Then
$$D'=\partial^2x(1-x)+\partial(X'-xU')+V',$$
\begin{equation*}
\begin{split}
V'&=MVM^{-1}=(\alpha+1)(1+u)\begin{pmatrix}-1&1\\1&-1\end{pmatrix},\\
X'&=MXM^{-1}=\begin{pmatrix}\alpha+1&-1-u\\0&\alpha+2\end{pmatrix},\\
U'&=MUM^{-1}=\begin{pmatrix}2\alpha+3&-1-u\\-1-u&2\alpha+3\end{pmatrix}.
\end{split}
\end{equation*}
If we write $D'=\begin{pmatrix}D'_{11}&D'_{12}\\D'_{21}&D'_{22}\end{pmatrix}$, we have
\begin{equation*}
\begin{split}
D'_{11}&=\big(\partial^2x(1-x)+\partial(\alpha+1-(2\alpha+3)x)-(\alpha+1)(1+u)\big)E_{11},\\
D'_{12}&=\big(-\partial(1+u)(1-x)+(\alpha+1)(1+u)\big)E_{12},\\
D'_{21}&=\big(\partial(1+u)x+(\alpha+1)(1+u)\big)E_{21},\\
D'_{22}&=\big(\partial^2x(1-x)+\partial(\alpha+2-(2\alpha+3)x)-(\alpha+1)(1+u)\big)E_{22},
\end{split}
\end{equation*}
where $\{E_{ij}\}$ is the canonical basis of $\Mat_2(\CC)$.
Let $\{p_n\}_{n\in\mathbb{N}_0}$ and $\{q_n\}_{n\in\mathbb{N}_0}$ be the sequences of monic orthogonal polynomials associated, respectively, to $w_1$ and $w_2$. Then $\{P_n=p_n\oplus q_n\}_{n\in\mathbb{N}_0}$ is the sequence of monic orthogonal polynomials with respect to $W'$.
If we use Proposition 2.7 in \cite{GT07} and $P_nD'=\Lambda'_nP_n$, then we get
$$\Lambda'_n=\begin{pmatrix}-n(n+2\alpha+2)-(\alpha+1)(1+u)&(1+u)(n+\alpha+1)\\
 (1+u)(n+\alpha+1)&-n(n+2\alpha+2)-(\alpha+1)(1+u)\end{pmatrix}.$$

In particular we obtain
$$x(1-x)p''_n+(\alpha+1-(2\alpha+3)x)p'_n+n(n+2\alpha+2)p_n=0.$$
As an immediate consequence of the hypergeometric representation of Jacobi polynomials
$$P^{(\alpha,\beta)}_n(y)=\frac{(\alpha+1)_n}{n!}\;_2F_1\left(\begin{matrix}-n,\,n+\alpha+\beta+1 \\\alpha+1\end{matrix};\frac{1-y}2\right)$$
we obtain
$$p_n(x)=\frac{(-1)^nn!}{(n+2\alpha+2)_n}P_n^{(\alpha,\alpha+1)}(1-2x),$$
$$q_n(x)=\frac{(-1)^nn!}{(n+2\alpha+2)_n}P_n^{(\alpha+1,\alpha)}(1-2x)$$
since the leading coefficient of $P^{(\alpha,\beta)}_n$ is $(n+\alpha+\beta+1)_n/2^n n!$.
From the identity $P_n^{\alpha,\beta}(-x)=(-1)^nP_n^{\beta,\alpha}(x)$ we get that
$$q_n(x)=(-1)^np_n(1-x).$$

It is wise and rewarding to verify the identity: $p_nD_{12}=\Lambda(D_{12})q_n$ which is equivalent to
\begin{equation*}\label{pq}
-(1-x)p'_n(x)+(\alpha+1)p_n=(n+\alpha+1)q_n(x).
\end{equation*}
By using the identity $\frac{d}{dx}P_n^{(\alpha,\beta)}=\frac12(n+\alpha+\beta+1)P_{n-1}^{(\alpha+1,\beta+1)}(x)$  we see that \eqref{pq} is equivalent to
\begin{equation*}
\begin{split}
(n+2\alpha+2)(1-x)P_{n-1}^{(\alpha+1,\alpha+2)}(1-2x)+(\alpha&+1)P_n^{(\alpha,\alpha+1)}(1-2x)\\
&=(n+\alpha+1)P_n^{(\alpha+1,\alpha)}(1-2x),
\end{split}
\end{equation*}
or to
\begin{equation*}
\begin{split}
n(n&+2\alpha+2)(1-x)\,_2F_1\left(\begin{matrix}-n+1,\,n+2\alpha+3 \\\alpha+2\end{matrix};x\right)\\
&+(\alpha+1)^2\,_2F_1\left(\begin{matrix}-n,\,n+2\alpha+2 \\\alpha+1\end{matrix};x\right)
=(n+\alpha+1)^2\,_2F_1\left(\begin{matrix}-n,\,n+2\alpha+2 \\\alpha+2\end{matrix};x\right).
\end{split}
\end{equation*}
By equating the coefficients of $x^j$ on both sides one can easily check that they are equal.

\

It is worth to point out now the following facts: $D'_{12}D'_{21}\in\mathcal D(w_1)$ and
$$\Lambda'_n(D'_{12}D'_{21})=\Lambda'_n(D'_{12})\Lambda'_n(D'_{21})=(1+u)^2(n+\alpha+1)^2.$$
Also $D'_{11}\in\mathcal D(w_1)$ and
$$\Lambda'_n(D'_{11})=-n(n+2\alpha+2)-(\alpha+1)(1+u).$$
If $u\ne -1$, then $\Lambda'_n(D'_{11}-(\alpha+1)(\alpha-u))=-(1+u)^{-2}\Lambda'_n(D'_{12}D'_{21})$. From here, and since the set of representations $\Lambda'_n$ ($n\in\mathbb{N}_0$) separates the elements of $\mathcal D(w_1)$ (Proposition 2.8,\cite{GT07}), the following interesting factorization follows
$$\partial^2x(1-x)+\partial(\alpha+1-(2\alpha+3)x)-(\alpha+1)^2=(\partial(1-x)-\alpha-1)
(\partial x+\alpha+1).$$
Similarly we have
$$\partial^2x(1-x)+\partial(\alpha+2-(2\alpha+3)x)-(\alpha+1)^2=(\partial x+\alpha+1)(\partial(1-x)-\alpha-1)
,$$
and these two second order differential operators are related by a Darboux transformation.
When $u\ne-1$ let
\begin{equation*}
\begin{split}
D_{11}&=\big(\partial^2x(1-x)+\partial(\alpha+1-(2\alpha+3)x)-(\alpha+1)^2\big)E_{11}\in\mathcal D(w_1),\\
D_{12}&=\big(-\partial(1-x)+\alpha+1\big)E_{12}\in\mathcal D(w_1,w_2),\\
D_{21}&=\big(\partial x+\alpha+1\big)E_{21}\in\mathcal D(w_2,w_1),\\
D_{22}&=\big(\partial^2x(1-x)+\partial(\alpha+2-(2\alpha+3)x)-(\alpha+1)^2\big)E_{22}\in\mathcal D(w_2).
\end{split}
\end{equation*}

\begin{prop} $\mathcal D(w_1\oplus w_2)$ can be presented as the complex algebra with generators $E_{11}, E_{22}, D_{12}$, $D_{21}$ and the relations derived from the matrix multiplication table $E_{ij}E_{rs}=\delta_{jr}E_{is}$. Moreover
$$\mathcal D(w_1\oplus w_2)=\begin{pmatrix}\mathcal D(w_1)&\mathcal D(w_1,w_2)\\
\mathcal D(w_2,w_1)&\mathcal D(w_2)\end{pmatrix}.$$
If we define $D_{11}=D_{12}D_{21}$ and $D_{22}=D_{21}D_{12}$, then $\mathcal D(w_1)$ is the polynomial algebra $\CC[D_{11}]$ and $\mathcal D(w_2)=\CC[D_{22}]$. While $\mathcal D(w_1,w_2)$ is the free left $\mathcal D(w_1)$-module (resp. the free right $\mathcal D(w_2)$-module  generated by $D_{12}$. Similarly $\mathcal D(w_2,w_1)$ is the free left $\mathcal D(w_2)$-module (resp. the free right $\mathcal D(w_1)$-module  generated by $D_{21}$. The center of $\mathcal D(w_1\oplus w_2)$ is the polynomial algebra $\CC[D_{11}+ D_{22}]$. Moreover, $\mathcal D(w_1\oplus w_2)$ is a free module over its center with basis
$\{E_{11},E_{22},D_{12},D_{21}\}$.
\end{prop}
\begin{proof}
In the classical cases of Bessel, Hermite, Jacobi and Laguerre scalar weights $w=w(x)$ the algebra $\mathcal D(w)$ is well understood: it is the polynomial algebra in any second order differential operator $\in\mathcal D(w)$ (cf. \cite{M05}). Therefore in our case
$\mathcal D(w_1)=\CC[D_{11}]$ and $\mathcal D(w_2)=\CC[D_{22}]$.

Now we will show that $\mathcal D(w_1,w_2)$ has no differential operators of even order. In fact, if $D\in\mathcal D(w_1,w_2)$ were of even order, then $DD_{21}\in\mathcal D(w_2)$ would be of odd order. This is a contradiction  which proves our assertion. Therefore we can assume that $D\in\mathcal D(w_1,w_2)$ is of the form
$$D=\sum_{i=0}^{2s+1}\partial^iF_i(x),\quad F_i(x)=\sum_{j=0}^ix^jF_j^i.$$
The coefficient of $\partial^{2(s+1)}$ in $DD_{21}$ is $xF_{2s+1}(x)$ whose leading coefficient is $F_{2s+1}^{2s+1}$. But the coefficient of a differential operator in $\mathcal D(w_2)$ of order $2(s+1)$ is a polynomial of degree $2(s+1)$. Therefore $F_{2s+1}^{2s+1}\ne0$.

By induction on the order of a differential operator $D\in\mathcal D(w_1,w_2)$ we will prove that $\mathcal D(w_1,w_2)=\mathcal D(w_1)D_{12}$. We already know that in $\mathcal D(w_1,w_2)$ there is no nonzero constant differential operator. Let $D\in\mathcal D(w_1,w_2)$ be of degree one. Then $D=\big(\partial(x F_1^1+F_0^1)+F_0^0\big)E_{12}$ with $F_1^1, F_0^1, F_0^0\in\CC$. We consider $D-F_1^1D_{12}\in\mathcal D(w_1,w_2)$. Since the coefficient of $\partial$ is constant it follows that $D-F_1^1D_{12}=0$. Now assume that $D\in\mathcal D(w_1,w_2)$ is of order $2s+1\ge3$. Then
$D=\sum_{i=0}^{2s+1}\partial^i F_i(x)$ with $F_{2s+1}(x)=\big(x^{2s+1}F_{2s+1}^{2s+1}+\cdots\big)E_{12}$, where the dots stand for lower degree terms and $F_{2s+1}^{2s+1}\in\CC^\times$.

We now consider $\tilde D=D-(-1)^sF_{2s+1}^{2s+1}D_{12}\in\mathcal D(w_1,w_2)$. By construction the degree of the polynomial coefficient of $\partial^{2s+1}$ in $\tilde D$ is of degree less than $2s+1$. Therefore $\mathrm{ord}(\tilde D)<2s+1$. By the inductive hypothesis there exists $D'\in\mathcal D(w_1)$ such that $\tilde D=D'D_{12}$, completing the proof that $\mathcal D(w_1,w_2)=\mathcal D(w_1)D_{12}$. Since it is clear that $D'D_{12}=0$ implies $D'=0$, we have proved that $\mathcal D(w_1,w_2)$ is a free left $\mathcal D(w_1)$-module generated by $D_{12}$.

By observing that $p(D_{11})D_{12}=p(D_{12}D_{21})D_{12}=D_{12}p(D_{22})$ for any $p\in\CC[x]$, it follows that $D(w_1,w_2)$ is the free right $\mathcal D(w_2)$-module  generated by $D_{12}$. In a similar way can be establish that $\mathcal D(w_2,w_1)$ is the free left $\mathcal D(w_2)$-module (resp. the free right $\mathcal D(w_1)$-module  generated by $D_{21}$.

Let $D$ be an element in the center of $\mathcal D(w_1\oplus w_2)$. If $D=p_{11}(D_{11})+
p_{12}(D_{11})D_{12}+p_{21}(D_{22})D_{21}+p_{22}(D_{22})$, $p_{11}, p_{12}, p_{21}, p_{22} \in \CC[x]$, then
$$DD_{12}=p_{11}(D_{11})D_{12}+p_{21}(D_{22})D_{22}\in\mathcal D(w_1,w_2)\oplus\mathcal D(w_2),$$
and
$$D_{12}D=D_{12}p_{21}(D_{22})D_{21}+D_{12}p_{22}(D_{22})\in\mathcal D(w_1)\oplus\mathcal D(w_1,w_2).$$
Hence $p_{21}(D_{22})=0$ and $p_{11}(D_{11})D_{12}=D_{12}p_{22}(D_{22})=p_{22}(D_{11})D_{12}$. Therefore $p_{11}=p_{22}$.
Similarly we have,
$$DD_{21}=p_{12}(D_{11})D_{11}+p_{22}(D_{22})D_{21}\in\mathcal D(w_1)\oplus\mathcal D(w_2,w_1),$$
and
$$D_{21}D=D_{21}p_{11}(D_{11})+D_{21}p_{12}(D_{11})D_{12}\in\mathcal D(w_2,w_1)\oplus\mathcal D(w_2).$$
We conclude that $p_{12}(D_{11})=0$. Therefore, if the zero order term of $D$ is zero, then  $D=p_{11}(D_{11})+p_{11}(D_{22})=p_{11}(D_{11}+D_{22})$, because $D_{11}D_{22}=D_{22}D_{11}=0$. Thus in any case $D=p(D_{11}+D_{22})$, $p\in\CC[x]$. Conversely, if $D=p(D_{11}+D_{22})$, then it is obvious that $D$ commutes with all elements of $\mathcal D(w_1)$ and of $\mathcal D(w_2)$. Moreover, $DD_{12}=p(D_{11}+D_{22})D_{12}=p(D_{11})D_{12}=D_{12}p(D_{22})
=D_{12}p(D_{11}+D_{22})$. In the same way one proves that $DD_{21}=D_{21}D$.
So, the center of $\mathcal D(w_1\oplus w_2)$ is equal to $\CC[D_{11}+ D_{22}]$.

Finally we point out that,
\begin{equation*}\begin{split}
p(D_{11}+D_{22})E_{11}&=p(D_{11}),\hskip.95cm p(D_{11}+D_{22})E_{22}=p(D_{22}),\\
p(D_{11}+D_{22})D_{12}&=p(D_{11})D_{12},\quad p(D_{11}+D_{22})D_{21}=p(D_{22})D_{21}.
\end{split}
\end{equation*}
From here it is easy to see that $\{E_{11},E_{22},D_{12},D_{21}\}$ is a basis of $\mathcal D(w_1\oplus w_2)$ over the polynomial ring $\CC[D_{11}+D_{22}]$. The proposition is proved. \qed
\end{proof}
\end{section}

\begin{section}{Criteria}

In this final section we give some criteria to decide whether a weight is irreducible or not.
First we state the following elementary lemma.
\begin{lem}\label{moments} Let $w=w(x)$ be an integrable function in a finite interval $(a,b)$
such that $x^kw(x)$ are integrable for all $k\in\NN_0$. If
$$\int_a^b x^kw(x)\,dx=0,\quad\text{for all}\quad k\in\NN_0,$$
then $w=0$.
\end{lem}
\begin{proof}
By hypothesis, if $p\in\CC[x]$ then $\int_a^b p(x)w(x)\,dx=0$. Given $f\in C([a,b])$ and $\epsilon>0$, by Stone-Weirstrass theorem, there exists $p\in\CC[x]$ such that
$\vert f(x)-p(x)\vert<\epsilon$ for all $x\in[a,b]$. Thus,
$$\Big\vert\int_a^b f(x)w(x)\,dx\Big\vert\le\int_a^b\vert f(x)-p(x)\vert \vert w(x)\vert\,dx<\epsilon\int_a^b\vert w(x)\vert\,dx.$$
Therefore $\int_a^b f(x)w(x)\,dx=0$. The Borel measure defined by $w$ is regular, hence by Riesz representation theorem $w=0$. \qed
\end{proof}

\

In the next theorem we give an alternative description of the real vector space $\mathcal C_\RR$, introduced in \ref{CR}, of a matrix weight $W$ defined on a finite interval $(a,b)$.

\begin{thm}\label{criterio1}  If $W$ is a matrix weight of size $N$ on a finite interval $(a,b)$, let $\{P_n\}_{n\in\mathbb{N}_0}$ be the sequence of monic orthogonal polynomials and let $M_0$ be the moment of order zero of $W$. Then
 $$\mathcal C_\RR=\{T: TM_0=M_0T^*, \,TP_n=P_nT \text{ for all } n\in\mathbb{N}_0\}.$$
\end{thm}
\begin{proof} If $T\in\mathcal C_\RR$, then $TW(x)=W(x)T^*$ for a.e. $x\in(a,b)$. Since $P_0=I$ it is obvious that $TP_0=P_0T$.
 Now assume that $n\ge1$ and that $TP_j(x)=P_j(x)T$ for all $x$ and all $0\le j\le n-1$. Then
 \begin{equation}
 \begin{split}
 (P_nT,P_j)&=\int_a^b P_n(x)TW(x)P_j^*(x)dx= \int_a^b P_n(x)W(x)T^*P_j^*(x)dx\\
 &=(P_n,P_jT)=(P_n,TP_j)=(P_n,P_j)T^*=0.
 \end{split}
 \end{equation}
 Hence $P_n(x)T=AP_n(x)$ for some $A\in\Mat_N(\CC)$ and all $x$. Since $P_n$ is monic it follows that $T=A$. Therefore by induction on $n$ we have proved that $P_nT=TP_n$ for  all $n\in\mathbb{N}_0$. Moreover, it is clear that $TM_0=M_0T^*$. Therefore
$$\mathcal C_\RR\subseteq\{T: TM_0=M_0T^*, \,TP_n=P_nT \text{ for all } n\in\mathbb{N}_0\}.$$

 To prove the reverse inclusion we start with a matrix $T$ such that  $TM_0=M_0T^*$ and
 $TP_n=P_nT$ for all  $n\in\mathbb{N}_0$. We will first prove by induction on $n\in\mathbb{N}_0$ that  $TM_n=M_nT^*$. Thus assume that $n\ge1$ and that $TM_j=M_jT^*$ for all $0\le j\le n-1$.  If we write $P_n=x^nI+x^{n-1}A_{n-1}+\cdots+A_0$, then by hypothesis $TA_j=A_jT$ for all $0\le j\le n-1$.

 Now $(P_n,P_0)=0$ is equivalent to $M_n+A_{n-1}M_{n-1}+\cdots +A_0M_0=0$.
 Hence $TM_n=M_nT^*$. Therefore
 $$\int_a^b x^n(TW(x)-W(x)T^*)\,dx=0\quad \text{for all }\quad n\in\mathbb{N}_0,$$
 which is equivalent to $TW(x)=W(x)T^*$ for almost every $x\in(a,b)$, see Lemma \ref{moments}. Thus
$$\{T: TM_0=M_0T^*, \,TP_n=P_nT \text{ for all } n\in\mathbb{N}_0\}\subseteq \mathcal C_\RR,$$
completing the proof of the theorem.\qed
\end{proof}

\

\begin{thm}\label{1}
If a matrix weight $W$ supported on $(a,b)$ reduces, then $\mathcal D(W)$ contains non-scalar operators of order zero. If $a,b\in\mathbb R$, then the converse is also true.
\end{thm}
\begin{proof}
Assume that $W$ reduces, then there exists an inversible matrix $M$ such that
\[
W(x)=
M\left(\begin{matrix}
W_1(x) & {\bf0} \\
\bf0 &W_2(x)
\end{matrix}\right)M^*,\qquad \text{for every }x\in(a,b),
\]
with $W_1,W_2$ weights of smaller size. Then, it is clear that \[M\left(\begin{matrix}
\bf1 & {\bf0} \\
\bf0 &\bf0
\end{matrix}\right)M^{-1}
\]
is an operator in $\mathcal D(W)$.
For the second statement, assume that $a,b\in\mathbb R$ and that there is a non-scalar order zero operator $D\in\mathcal D(W)$.
By \eqref{real}, we have that $D=T+iT'$, with $T,T'\in\mathcal D(W)$ symmetric differential operators. Since $D$ is non-scalar then $T$ or $T'$ is non-scalar.
Assume that $T$ is non-scalar and let $\{P_n\}_{n\in\mathbb{N}_0}$ be the sequence of monic orthogonal polynomials with respect to $W$. Clearly, $TP_n=P_nT$ for $n\in\mathbb N$.
Since the adjoint of $T$ is given by $\widetilde T=\,\|P_0\|^2T^*\,\|P_0\|^{-2}$ (\cite[Theorem 4.3]{GT07}) we have that $T\in\mathcal C_\mathbb{R}$, see Theorem \ref{criterio1}. If $T$ is scalar, then $T'$ is non-scalar. Analogously, it can be seen that $T'\in\mathcal C_\mathbb{R}$. Hence, $\mathcal C_\mathbb{R}$ is non-trivial. Therefore, by Theorem \ref{reducibility1}, $W$ reduces.\qed
\end{proof}
\begin{cor}\label{cor}
 If a matrix weight $W$ on $(a,b)$ reduces, then the commutant of the monic orthogonal polynomials $\{P_n\}_{n\in\mathbb N_0}$, given by
 $$\{T\in\Mat_N(\CC):TP_n(x)=P_n(x)T, \text{ for all }\; n, x\},$$  contains non-scalar matrices. If $a,b\in\mathbb R$, then the converse is also true.
\end{cor}

\begin{thm}\label{ab}
Let $W$ be a matrix weight  on $(a,b)$ and
$$
 A_n P_{n-1}(x)+B_n P_n(x)+P_{n+1}(x)=x P_n(x)
$$
be  the three term recursion relation
satisfied by the corresponding sequence of monic orthogonal polynomials $\{P_n\}_{n\in\mathbb{N}_0}$. If $W$ reduces, then there is a non-scalar matrix commuting with $\{A_n,B_n\}_{n\in\mathbb{N}_0}$. Even more, if $a,b\in\mathbb R$, then the converse is also true.
\end{thm}
\begin{proof}
If $W$ reduces then, by Corollary \ref{cor}, there is non-scalar matrix $T$  commuting with $P_n$ for every $n$, then we have
 $$
 A_n T P_{n-1}(x)+B_nTP_n(x)+TP_{n+1}(x)=TA_nP_{n-1}(x)+TB_nP_n(x)+TP_{n+1}(x),
$$
which implies that, for every $n$, $A_nT=TA_n$ and $B_nT=TB_n$.

 Assume $a,b\in\mathbb R$ and let   $T$ be a matrix commuting with $A_n$ and $B_n$ for all $n\in\mathbb{N}_0$. Then, we have
 $$A_nP_{n-1}(x)T+B_nP_n(x)T+P_{n+1}(x)T=xP_n(x)T,$$
 $$A_nTP_{n-1}(x)+B_nTP_n(x)+TP_{n+1}(x)=x TP_n(x),$$
for every $n$. Since $P_0(x)T=TP_0(x)$, by induction on $n\ge0$ it is clear that $P_n(x)T=TP_n(x)$ for all $n\in\mathbb{N}_0$.
By  Corollary \ref{cor} the proof is finished.\qed
\end{proof}

\begin{cor}
 Let $\{Q_n\}_{n\in\mathbb{N}_0}$ be a sequence of orthogonal polynomials with respect to a weight $W$ and  with respect to a weight $V$, both on a finite interval $(a,b)$. Then, $W$ is irreducible if and only if $V$ is irreducible.
\end{cor}

We feel that it is time to introduce the following definition.
\begin{defn}\label{def}
Given a sequence of matrix polynomials $\{Q_n\}_{n\in\NN_0}$ we say that it is reducible if it is equivalent to a sequence $\{\tilde Q_n\}_{n\in\NN_0}$
(i.e. $\tilde Q_n(x)=MQ_n(x)M^{-1}$, for all $x, n$ and some non-singular constant matrix $M$)
such that
$$\tilde Q_n(x)=\left(\begin{matrix}\tilde  Q_n^1(x) & \bf0\\ \bf0&\tilde  Q_n^2(x)\end{matrix}\right),\quad \text{ for all } n\in\NN_0.$$
\end{defn}
For more general definitions of similarity and equivalence of matrix polynomials the reader can see \cite{FN14}.

\begin{lem}\label{Favard} If $\{Q_n\}_{n\in\NN_0}$ is a sequence of matrix orthogonal polynomials equivalent to $\{\tilde Q_n=\tilde Q_n^1\oplus\tilde Q_n^2\}_{n\in\NN_0}$, then $\{\tilde Q_n^1\}_{n\in\NN_0}$ and $\{\tilde Q_n^2\}_{n\in\NN_0}$ are sequences of matrix orthogonal polynomials.
\end{lem}
\begin{proof}Let $W$ be the orthogonality weight function of $\{Q_n\}_{n\in\NN_0}$. Let $M$ be a matrix such that $\tilde Q_n=MQ_nM^{-1}$, for every $n\in \NN_0$. Thus, the  sequence  $\{\tilde Q_n\}_{n\in\NN_0}$ is orthogonal with respect to $\tilde W=MWM^*$. Let
$$\tilde W=\left(\begin{matrix} \tilde W_{1,1} & \tilde W_{1,2}\\ \tilde W_{2,1} & \tilde W_{2,2}\end{matrix}\right),$$
it is clear that $\tilde W_{1,1}$ and $\tilde W_{2,2}$ are also matrix weights.

A direct calculation gives
$${\bf0}=\langle \tilde Q_m,\tilde Q_n\rangle=\left(\begin{matrix}\langle\tilde Q_m^1,\tilde Q_n^1\rangle_{\tilde W_{1,1}} & \langle\tilde Q_m^1,\tilde Q_n^2\rangle_{\tilde W_{1,2}}\\ \langle\tilde Q_m^2,\tilde Q_n^1\rangle_{\tilde W_{2,1}} & \langle\tilde Q_m^2,\tilde Q_n^2\rangle_{\tilde W_{2,2}}\end{matrix}\right).$$
Therefore $\{\tilde Q_n^1\}_{n\in\NN_0}$ and $\{\tilde Q_n^2\}_{n\in\NN_0}$
are, respectively, orthogonal sequences with respect to the weight functions  $\tilde W_{1,1}$ and $\tilde W_{2,2}$. \qed
\end{proof}

\begin{thm}
 A sequence of monic orthogonal polynomials $\{P_n\}_{n\in\mathbb{N}_0}$ with respect to a weight $W$ is reducible if and only if
its commutant, given by
 $$\{T\in\Mat_N(\CC):TP_n(x)=P_n(x)T, \text{ for all }\; n, x\},$$ contains non-scalar matrices.
\end{thm}
\begin{proof}Necessity follows from Lemma \ref{Favard} and Corollary \ref{cor}. For the sufficiency  we can assume, without loss of generality, that $\|P_0\|=I$.
We have that  $\{P_n\}_{n\in\mathbb{N}_0}$ satisfy a three term recursion relation
\begin{equation}\label{threeterm}
 A_n P_{n-1}(x)+B_n P_n(x)+P_{n+1}(x)=x P_n(x),
\end{equation}
 for $n\in\mathbb{N}_0$. Let $S_n=\|P_n\|^2$. Then it is well known that \eqref{threeterm} implies that
\begin{equation}\label{uu}
S_n=A_nS_{n-1},\qquad
B_nS_n=(B_nS_n)^*.
\end{equation}

By hypothesis, there is a non-scalar matrix $T$ commuting with $P_n$ for every $n\in\NN_0$. From  the proof of Theorem \ref{ab}, it is clear that $T$ also commutes with $A_n$ and $B_n$ for all  $n\in\mathbb N_0$. Since $S_0$ is the identity, it follows from \eqref{uu} that $T$ also commutes with $S_n$ for all $n\in\NN_0$. Therefore, every eigenspace of $T$ is invariant by $\{P_n, A_n, B_n, S_n: n\in\NN_0\}$. Thus, we have that there exists a unitary matrix $M$ such that, for every $n\in\NN_0$,
\begin{align*}
 MP_n(x)M^{-1}=
&
\left(\begin{matrix}P_n^1(x)&{P_n^0(x)}\\ {\bf0}&P_n^2(x)\end{matrix}\right),
&
MS_nM^{-1}
&
=\left(\begin{matrix}S_n^1&{S_n^0}\\ {\bf0}&S_n^2\end{matrix}\right),
\\
MA_nM^{-1}=
&
\left(\begin{matrix}A_n^1&{A_n^0}\\ {\bf0}&A_n^2\end{matrix}\right),
&
MB_nM^{-1}
&
=
\left(\begin{matrix}B_n^1&{B_n^0}\\ {\bf0}&B_n^2\end{matrix}\right),
\end{align*}
where $P_n^j(x),A_n^j,B_n^j,S_n^j$ are square matrices of size $d_j$, for $j=1,2$, and $P_n^0(x),A_n^0,$ $B_n^0,S_n^0$ are $d_1\times d_2$ matrices. Since $S_n$ is positive definite we have that $S_n^1,S_n^2$ are positive definite and $S_n^0=\bf0$. By \eqref{uu}, $ B_nS_n$ is Hermitian, thus $B_n^0=\bf0$; also from \eqref{uu}, it is clear that $A_n^0=\bf0$.

By conjugating \eqref{threeterm} by $M$ and looking at the upper-right  blocks  we see that
\begin{equation*}
 A_n^1 P_{n-1}^0(x)+B_n^1 P_n^0(x)+P_{n+1}^0(x)=x P_n^0(x),
\end{equation*}
for  $n\in\mathbb{N}_0$.
Since $P_0=\bf I$ we have $P_0^0=\bf0$. Then, by induction on $n$ we get $P_n^0=0$ for every $n\in\NN_0$. Thus $MP_nM^{-1}=P_n^1\oplus P_n^2$, completing the proof of the theorem.
\qed
\end{proof}

\begin{cor}
  A sequence   of orthogonal polynomials $\{Q_n\}_{n\in\mathbb{N}_0}$ with respect to a weight $W$ on $(a,b)$ is orthogonal with respect to a reducible weight if and only if
there is a non-scalar operator of order zero in $\mathcal D(W)$.
\end{cor}
\begin{lem}\label{I}
Let $\{Q_n\}_{n\in\mathbb{N}_0}$ be a sequence of orthogonal polynomials with respect to a matrix weight $W$ and with respect to a  matrix  weight $V$. If
\[
\int_\RR W(x)dx =\int_\RR V(x)dx,\\
\]
 then $W$ and $ V$ have the same moments. If both weights have  compact  support, then $W=V$.
\end{lem}
\begin{proof}
Given the sequence $\{Q_n\}_{n\in\mathbb{N}_0}$, let us consider the corresponding sequence of monic  orthogonal  polynomials $\{P_n\}_{n\in\mathbb{N}_0}$. This sequence is also orthogonal with respect to $W$ and $V$, therefore they
satisfy a  three term  recursion relation
\begin{equation*}
 A_n P_{n-1}(x)+B_n P_n(x)+P_{n+1}(x)=x P_n(x)
\end{equation*}
   for $n\in\mathbb{N}_0$.
 From the hypothesis we have that $\|P_0\|_W^2=\|P_0\|_V^2$. Besides,  from \eqref{uu}  we have
$\|P_{n+1}\|^2=A_{n+1}\|P_{n}\|^2$ for every $n\in\mathbb{N}_0$ and
for both norms  $\|\cdot\|_W^2$ and $\|\cdot\|_V^2$. Therefore $\|P_n\|_W^2=\|P_n\|_V^2$ for any $n\in\mathbb{N}_0$. Hence,
$$\int_\RR P_n(x)(W(x)-V(x))P_m^*(x)dx =0,$$
 for every $n,m\ge 0$. Hence, $W$ and $V$ have the same moments. If both have compact support, then  by Lemma~\ref{moments} we have $W=V$.\qed
 \end{proof}

\begin{thm}\label{scalar}
Let $\{Q_n\}_{n\in\mathbb{N}_0}$ be a  sequence of orthogonal polynomials with respect to an irreducible weight $W$ and  with respect to a weight $V$, both with compact support. Then, $V$ is a scalar multiple of $W$.
\end{thm}

\begin{proof}Let us consider the corresponding sequence of monic orthogonal polynomials $\{P_n\}_{n\in\mathbb{N}_0}$.
From \eqref{uu} we have that $\|P_{n+1}\|^2=A_{n+1}\|P_{n}\|^2$, for  $n\in\NN_0$ and the inner products given by $W$ or $V$. As a straightforward consequence we obtain
\begin{equation*}
 A_n\cdots A_1\|P_0\|^2_W=\|P_n\|^2_W,\quad
 A_n\cdots A_1\|P_0\|^2_V=\|P_n\|^2_V,
\end{equation*}
for $\in\NN_0$. Hence, since the norms are nonsingular hermitian matrices, we have
$$
\|P_n\|^2_W\,\|P_0\|^{-2}_W=\|P_n\|^2_V\,\|P_0\|^{-2}_V\quad\text{and }\quad
\|P_0\|^{-2}_W\,\|P_n\|^2_W=\|P_0\|^{-2}_V\,\|P_n\|^2_V.
$$
Then,
\begin{multline*}
 A_n\dots A_1\|P_0\|^2_V=
\|P_0\|^2_V\,\|P_0\|^{-2}_V\,\|P_n\|^2_V=
\|P_0\|^2_V\,\|P_0\|^{-2}_W\,\|P_n\|^2_W\\
=
\|P_0\|^2_V\,\|P_0\|^{-2}_W\,A_n\dots A_1\|P_0\|^2_W.
\end{multline*}
Hence, $\|P_0\|^2_V\,\|P_0\|^{-2}_W$ commutes with $A_n\cdots A_1$ for every $n\in\NN_0$. Since $A_n$ is nonsingular for all $n\ge1$ it follows that $\|P_0\|^2_V\,\|P_0\|^{-2}_W$ commutes with $A_n$ for all $n\in\NN_0$.

For $n\in\NN_0$ and the inner products given by $W$ or $V$ we have $B_n\|P_n\|^2=(B_n\|P_n\|^2)^*$, see \eqref{uu}.
Thus
\begin{align*}
 B_n\|P_n\|^2_W&=(B_n\|P_n\|^2_W)^*=
 (B_n \|P_n\|^2_V\,\|P_0\|^{-2}_V\,\|P_0\|^2_W)^*\\
 &=\|P_0\|^2_W\,\|P_0\|^{-2}_V(B_n \|P_n\|^2_V)^*
= \|P_0\|^2_W\,\|P_0\|^{-2}_VB_n \|P_n\|^2_V\\
  &=\|P_0\|^2_W\,\|P_0\|^{-2}_VB_n\,\|P_0\|^2_V\,\|P_0\|^{-2}_W\, \|P_n\|^2_W.
\end{align*}
Therefore, $\|P_0\|^2_V\|P_0\|^{-2}_W$ also commutes with $B_n$ for every $n\in\NN_0$.
By assumption $W$ is irreducible, thus Theorem \ref{ab} implies that $\|P_0\|^2_V=\lambda\|P_0\|^{2}_W$, with $\lambda{>0}$. Finally, Lemma \ref{I} completes the proof of the theorem.\qed
\end{proof}

\smallskip
When the support is bounded, we have proved that the reducibility of a weight is equivalent to the reducibility of any of its sequence of orthogonal polynomials. Even more, we showed that there is only one, up to scalar, irreducible weight for a given irreducible sequence of orthogonal polynomials. When the support is not bounded, this is not true; this is one of the main motivations behind Definition \ref{def}.
We give the following ad hoc example to help the reader to understand this.

\begin{ex}\label{mom}
Let us consider the weight $w$ defined by
$$w(x)=\frac2{x}e^{-(\log x)^2/2}, \quad\text{for}\quad x>0,$$
and let $\{p_n\}_{n\in\mathbb N_0}$ be a sequence of orthogonal polynomials with respect to $w$.
Let $$f(x)=\frac{1}{x}e^{-(\log x)^2/2}\sin(2\pi\log x),\quad\text{for}\quad x>0.$$
The function $x^kf(x)$ is continuous in $(0,\infty)$ and
$\vert x^kf(x)\vert\le\frac{x^k}{x}e^{-(\log x)^2/2}$. By making the change of variables $\log x=s+k$ we get,
$$\int_0^\infty x^k e^{-(\log x)^2/2}\frac{dx}{x}=e^{k^2/2}\int_{-\infty}^\infty e^{-s^2/2}\,ds=\frac1{\sqrt{2\pi}}e^{k^2/2}.$$
Hence $x^kf(x)$ is an integrable function in $(0,\infty)$, for all $k\ge0$. Moreover,	
\begin{equation*}
\begin{split}
\int_0^\infty x^kf(x)\,dx&=\int_0^\infty x^k e^{-(\log x)^2/2}\sin(2\pi\log x)\frac{dx}{x}\\
&=e^{k^2/2}\int_{-\infty}^\infty e^{-s^2/2}\sin(2\pi s)\,ds=0.
\end{split}
\end{equation*}
Now, we define the matrix weight $W$ by
$$W(x)=\left(\begin{matrix}
        2 w(x)&f(x)\\f(x)&w(x)
             \end{matrix}
\right), \quad\text{for}\quad x>0.$$
The reader can easily check that $W(x)$ is positive definite for all $x>0$. Also, by Theorem \ref{reducibility1}, after some straight forward computations, it follows that $W$ is irreducible. Then, the sequence of polynomials $\{P_n\}_{n\in\mathbb N_0}$, given by
$$P_n(x)=\left(\begin{matrix}
        p_n(x)&0\\0&p_n(x)
             \end{matrix}
\right),$$
is orthogonal with respect to the irreducible weight $W$. Clearly, $\{P_n\}_{n\in\mathbb N_0}$ is also orthogonal with respect to the matrix weight
$$V(x)=\left(\begin{matrix}
        w(x)&0\\0&w(x)
             \end{matrix}
\right), \quad\text{for}\quad x>0.$$
\end{ex}

\end{section}

%
%

\begin{thebibliography}{AKdlR15}

\bibitem{AKR15}
N.~Aldenhoven, E.~Koelink, and A.~M. de~los R{\'{\i}}os, \emph{Matrix-valued
  little {$q$}-{J}acobi polynomials}, J. Approx. Theory \textbf{193} (2015),
  164--183. 

\bibitem{B29}
S.~{Bochner}, \emph{{\"Uber Sturm-Liouvillesche Polynomsysteme}}, {Math. Z.}
  \textbf{29} (1929), 730--736 (German).

\bibitem{CI14}
M.~Cafasso and M.~D. de~la Iglesia, \emph{Non-commutative {P}ainlev\'e
  equations and {H}ermite-type matrix orthogonal polynomials}, Comm. Math.
  Phys. \textbf{326} (2014), no.~2, 559--583. 

\bibitem{CG05}
M.~M. Castro and F.~A. Gr{\"u}nbaum, \emph{{Orthogonal matrix polynomials
  satisfying first order differential equations: a collection of instructive
  examples}}, J. Nonlinear Math. Physics \textbf{12} (2005), no.~2, 63--67.

\bibitem{CG06}
M.~M. Castro and F.~A. Gr{\"u}nbaum, \emph{{The algebra of differential operators associated to a given
  family of matrix valued orthogonal polynomials: five instructive examples}},
  Int. Math. Res. Not. \textbf{27} (2006), no.~2, 1--33.

\bibitem{C99}
C.~Chevalley, \emph{Theory of lie groups}, Princeton University Press,
  Princeton, 1999.

\bibitem{DdI08}
A.~J. Dur\'an and M.D. de~la Iglesia, \emph{Some examples of orthogonal matrix
  polynomials satisfying odd order differential equations}, Journal of
  Approximation Theory \textbf{150} (2008), no.~2, 153--174.

\bibitem{DG04}
A.~J. Dur\'an and F.~A. Gr{\"u}nbaum, \emph{Orthogonal matrix polynomials
  satisfying second-order differential equations}, Int. Math. Res. Not. \textbf{10} (2004),
  461--484.

\bibitem{DG05a}
A.~J. Dur\'an and F.~A. Gr{\"u}nbaum, \emph{{A characterization for a class of weight matrices with
  orthogonal matrix polynomials satisfying second-order differential
  equations.}}, {Int. Math. Res. Not.} \textbf{23} (2005), 1371--1390.

\bibitem{I11}
M.~D. de~la Iglesia, \emph{Some examples of matrix-valued orthogonal functions
  having a differential and an integral operator as eigenfunctions}, J. Approx.
  Theory \textbf{163} (2011), no.~5, 663--687. 

\bibitem{D97}
A.~J. Dur\'an, \emph{{Matrix inner product having a matrix symmetric
  second-order differential operator}}, Rocky Mt. J. Math. \textbf{27} (1997),
  no.~2, 585--600.

\bibitem{FN14}
Karl-Heinz F{\"o}rster and B{\'e}la Nagy, \emph{Equivalences of matrix
  polynomials}, Acta Sci. Math. (Szeged) \textbf{80} (2014), no.~1-2, 233--260.
  

\bibitem{GIV11}
F.~A. Gr{\"u}nbaum, M.~D. de~la Iglesia, and A.~Mart{\'{\i}}nez,
  \emph{Properties of matrix orthogonal polynomials via their
  {R}iemann-{H}ilbert characterization}, SIGMA Symmetry Integrability Geom.
  Methods Appl. \textbf{7} (2011), Paper 098, 31. 

\bibitem{GPT02a}
F.~A. Gr{\"u}nbaum, I.~Pacharoni, and J.~Tirao, \emph{Matrix valued spherical
  functions associated to the complex projective plane}, J. Funct. Anal.
  \textbf{188} (2002), no.~2, 350--441.

\bibitem{GPT03}
F.~A. Gr{\"u}nbaum, I.~Pacharoni, and J.~Tirao, \emph{Matrix valued orthogonal polynomials of the {J}acobi type},
  Indag. Math. (N.S.) \textbf{14} (2003), no.~3-4, 353--366.

\bibitem{GPT05}
F.~A. Gr{\"u}nbaum, I.~Pacharoni, and J.~Tirao, \emph{Matrix valued orthogonal polynomials of {J}acobi type: the role
  of group representation theory}, Ann. Inst. Fourier (Grenoble) \textbf{55}
  (2005), no.~6, 2051--2068.

\bibitem{G03}
F.~A. Gr{\"u}nbaum, \emph{Matrix valued {J}acobi polynomials}, Bull. Sci. Math.
  \textbf{127} (2003), no.~3, 207--214.

\bibitem{GT07}
F.~A. Gr{\"u}nbaum and J.~Tirao, \emph{The algebra of differential operators
  associated to a weight matrix}, Integral Equations Operator Theory
  \textbf{58} (2007), no.~4, 449--475.

\bibitem{HK65}
K.~Hoffman and R.~Kunze, ``{Linear algebra}'', Prentice--Hall, Inc., Englewood
  Cliffs, N. J., 1965.

\bibitem{KR15}
E.~Koelink and P.~Rom\'an, \emph{Orthogonal vs. non-orthogonal reducibility of
  matrix-valued measures}, arXiv:1509.06143 (2015).

\bibitem{K49}
M.~G. Krein, \emph{Infinite j-matrices and a matrix moment problem}, Dokl.
  Akad. Nauk SSSR \textbf{69} (1949), no.~2, 125--128.

\bibitem{K71}
M.~G. Krein, \emph{Fundamental aspects of the representation theory of hermitian
  operators with deficiency index $(m,m)$}, AMS Translations, series 2
  \textbf{97} (1971), 75--143.

\bibitem{KvPR12}
E.~Koelink, M.~van Pruijssen, and P.~Rom{\'a}n, \emph{Matrix-valued orthogonal
  polynomials related to {$({\rm SU}(2)\times{\rm SU}(2),{\rm diag})$}}, Int.
  Math. Res. Not. IMRN (2012), no.~24, 5673--5730. 

\bibitem{M05}
L.~Miranian, \emph{Matrix-valued orthogonal polynomials on the real line: some
  extensions of the classical theory}, J. Phys. A \textbf{38} (2005), no.~25,
  5731--5749. 

\bibitem{PT07a}
I.~Pacharoni and J.~Tirao, \emph{{Three term recursion relation for spherical
  functions associated to the complex hyperbolic plane}}, J. Lie Theory
  \textbf{17} (2007), no.~4, 791--828.

\bibitem{PTZ14}
I.~{Pacharoni}, J.~{Tirao}, and I.~{Zurri\'an}, \emph{{Spherical functions
  associated with the three-dimensional sphere.}}, {Ann. Mat. Pura Appl. (4)}
  \textbf{193} (2014), no.~6, 1727--1778 (English).

\bibitem{T77}
J.~Tirao, \emph{{Spherical functions}}, Rev. Un. Mat. Argentina \textbf{28}
  (1977), 75--98.

\bibitem{T03}
J.~Tirao, \emph{The matrix-valued hypergeometric equation}, Proc. Natl. Acad.
  Sci. U.S.A. \textbf{100} (2003), no.~14, 8138--8141.

\bibitem{T11}
J.~Tirao, {{The algebra of differential operators associated to a weight
  matrix: a first example}}, {Polcino Milies, C\'esar (ed.), ``Groups, algebras
  and applications. XVIII Latin American algebra colloquium, S\~ao Pedro,
  Brazil, August 3--8, 2009. Proceedings. Providence'', RI: American Mathematical
  Society (AMS). Contemporary Mathematics 537,  2011, 291-324}.

\bibitem{TZ14}
J.~A. {Tirao} and I.~{Zurri\'an}, \emph{{Spherical functions: the spheres
  versus the projective spaces.}}, {J. Lie Theory} \textbf{24} (2014), no.~1,
  147--157 (English).

\bibitem{vPR14}
M.~van {Pruijssen} and P.~Rom{\'a}n, \emph{Matrix valued classical pairs
  related to compact {G}elfand pairs of rank one}, SIGMA Symmetry Integrability
  Geom. Methods Appl. \textbf{10} (2014), Paper 113, 28. 

\bibitem{Z15}
I.~{Zurri{\'a}n}, \emph{{The Algebra of Differential Operators for a Gegenbauer
  Weight Matrix}}, arXiv:1505.03321  (2015).

\end{thebibliography}
\end{document}